\documentclass[11pt]{amsart}
\usepackage{amsopn}
\usepackage{amssymb}
\usepackage{graphicx}
\usepackage[all]{xy}
\usepackage{amscd}
\usepackage{color}
\usepackage[colorlinks]{hyperref}
 \newtheorem{thm}{Theorem}[section]

 \newtheorem{exam}[thm]{Example}
 
 \theoremstyle{definition}
 \theoremstyle{remark}
\begin{document}
\title[$*$-multiplication operators]
{ Viewing $*$-multiplication operators between Orlicz spaces }
\author{jahangir cheshmavar}
\author{seyed kamel hosseini}

\email{j$_{_-}$cheshmavar@pnu.ac.ir}
\email{kamelhosseini@chmail.ir}

\address{faculty of mathematics, payame noor  university,
p. o. box 19395-3697, tehran, iran.}

\thanks{}
\thanks{}
%\date{\today}
\dedicatory{}
\commby{}
\begin{abstract}
In this paper, Lambert multipliers acting between Orlicz spaces
are characterized based on some properties of conditional
expectation operators. Also, we provide a necessary and sufficient
condition for the $*$-multiplication operators to have closed
range. Finally, the necessary condition for Fredholmness of these
type operators are investigated.
\end{abstract}
\maketitle

%======================================================================

\section{introduction and preliminaries}\label{Sec1}

Operator in function spaces defined by conditional expectations
appeared already in Chen and Moy paper ~\cite{Chen.Moy} and Sidak
~\cite{Sidak}, see for example Brunk ~\cite{Brunk}, in the setting
of $L^p$ spaces. This class of operators was further studied by
Lambert ~\cite{Alan.1, Alan.2, Alan.3}, Herron ~\cite{Herron},
Takagi and Yokouchi ~\cite{ Takagi.1999}. Later, Jabbarzadeh and
Sarbaz in ~\cite{Jabbarzadeh.2010} have characterized the Lambert
multipliers acting between two $L^p$-spaces, by using some
properties of conditional expectation operator. Also,
Multiplication operators have been a subject of research of many
mathematicians, see for instance, ~\cite{Raj, Seid, Sharma}. In
this paper, Lambert multipliers on Orlicz spaces are considered.
Also, characterizations of existence and closedness of the range
of $*$-multiplication operators are provided.\\
 Let $\Phi: \mathbb{R} \rightarrow \mathbb{R}^+$ be a continuous
convex function satisfying the conditions: $\Phi(x)=\Phi(-x)$,
$\Phi(0)=0$ and $\lim_{x\rightarrow\infty} \Phi(x)= +\infty$,
where $\mathbb{R}$ denotes the set of real numbers. With each such
function $\Phi$, one can associate another convex function $\Psi :
\mathbb{R} \rightarrow \mathbb{R}^+$ having similar properties,
which is defined by $$\Psi(y)=\sup\{x|y|-\Phi(x): x\geq 0\},\,\
y\in \mathbb{R}~.$$ The function $\Phi$ is called a \textit{Young
function}, and $\Psi$ the \textit{complementary function} to
$\Phi$. A Young function $\Phi$ is said to satisfy the
$\bigtriangleup_{2}$-condition (globally) if $\Phi(2x)\leq
k\Phi(x), x\geq x_{0}\geq 0\,\ (x_{0}=0)$ for some absolute
constant $k>0$. The simple functions are not dense in
$L^\Phi(\Sigma)$ but if $\Phi$ satisfies the
$\bigtriangleup_{2}$-condition, then the class of simple functions
is dense in $L^\Phi(\Sigma)$.  Throughout this paper we assume
that $\Phi$ satisfies $\bigtriangleup_{2}$-condition. Throughout
this paper we assume that $\Phi$ satisfies
$\bigtriangleup_{2}$-condition.

Let $ (X,\Sigma,\mu)$ be a complete $\sigma$-finite measure space
and $\Phi$ be a Young function, then the set of
$\Sigma$-measurable functions
$$L^\Phi(\Sigma):=\{f:X\rightarrow \mathbb{C};
 \int_{X}\Phi(\varepsilon|f|)d\mu<\infty, \,\ \mbox{for some}\,\ \varepsilon>0\}~,$$ is
a Banach space, with respect to the norm
$$\|f\|_{\Phi}=\inf\{\varepsilon>0 : \int_{X}\Phi(|f|/\varepsilon)d\mu\leq
1\}~.$$ Such a space is known as a \textit{Orlicz space}.
For more details concerning Young function and Orlicz spaces,
 we refer to Labuschagne, Rao and Ren in \cite{Labuschagne, Rao}.\\
If $\Phi(x)=|x|^p, \,\ 1\leq p <
 \infty$, then $L^\Phi(\Sigma) = L^p(\Sigma)$, the usual $p$-integrable
 functions on $X$.

 Let $\mathcal{A}\subseteq \Sigma$ be complete
$\sigma$-finite subalgebra. We view
$L^p(\mathcal{A})=L^p(X,\mathcal{A},\mu|_{\mathcal{A}})$ as a
Banach subspace of $L^p(\Sigma)$. Denote the vector space of all
equivalence classes of almost everywhere finite valued $\Sigma$-
measurable functions on $X$, by $L^0(\Sigma)$. For each
nonnegative $f\in L^0(\Sigma)$ or $f\in L^p(\Sigma)$, by the
Radon-Nikodym theorem, there exists a unique
measurable function $E(f)$ with the following conditions:\\
(i) $E(f)$ is $\mathcal{A}$-measurable;\\
(ii) If $A$ is any $\mathcal{A}$-measurable set for which
$\int_Afd\mu$ converges, we have
$$\int_{A}fd\mu=\int_{A}E(f)d\mu~.$$
For every complete $\sigma$-finite subalgebra
$\mathcal{A}\subseteq \Sigma$, the mapping $f\longmapsto E(f)$,
from $L^p(\Sigma)$ to $L^p(\mathcal{A}), 1\leq p \leq \infty$, is
called the \textit{conditional expectation operator with respect
to $\mathcal{A}$}. As an operator on $L^p(\Sigma)$, $E(\cdot)$ is
contractive idempotent and $E(L^p(\Sigma))=L^p(\mathcal{A})$. We
will need the following standard facts concerning $E(f)$, for more
details of the properties of $E$, we refer the interested
 reader to \cite{Herron,Alan.2,Rao.1993}:\\
\begin{itemize}
  \item If $g$ is $\mathcal{A}$-measurable then $E(fg)=E(f)g$;
  \item $|E(f)|^p\leq E(|f|^p)$;
  \item $\|E(f)\|_p\leq\|f\|_p$;
  \item If $f\geq 0$ then $E(f)\geq 0$; if $f>0$ then $E(f)>0$;
  \item $E(1)=1$.
\end{itemize}

Let $\Phi$ be a Young function and $f \in L^\Phi(\Sigma)$, since
$\Phi$ is convex, by Jensen's inequality,
\begin{itemize}
  \item $\Phi(|E(f)|)\leq E(\Phi(|f|))$;
\end{itemize}
and consequently,
$$\int_{X}\Phi(\frac{E(f)}{\|f\|_{\Phi}})d\mu =\int_{X}\Phi(E(\frac{f}{\|f\|_{\Phi}}))
  \leq \int_{X}E(\Phi(\frac{f}{\|f\|_{\Phi}}))
  =\int_{X}\Phi(\frac{f}{\|f\|_{\Phi}}),$$
this implies that
\begin{itemize}
  \item $\|E(f)\|_\Phi\leq\|f\|_\Phi$,
\end{itemize}
that is, $E$ is contraction on Orlicz spaces. Now let $$D(E) :=
\{f \in L^0(\Sigma) : E(|f|) \in L^0(\mathcal{A})\}~,$$ then $f$
is said to be \textit{conditionable} with respect to $E$ if $f\in
D(E)$. For $f$ and $g$ in $D(E)$, we define $$f*g = fE(g)+gE(f)-
E(f)E(g)~.$$ Let $\Phi$ and $\Psi$ be Young functions. A
measurable function $u$ in $D(E)$ for which $u*f \in
L^{\Psi}(\Sigma)$ for each $f \in L^{\Phi}(\Sigma)$ is called a
\textit{Lambert multiplier}. In other words, $u \in D(E)$ is a
Lambert multiplier if and only if the corresponding
$*$-multiplication operator $T_u : L^{\Phi}(\Sigma) \rightarrow
L^{\Psi}(\Sigma)$ defined as $T_uf = u*f$ is bounded. Our
exposition regarding Lambert multipliers follows ~\cite{Herron,
Alan.3}.

As a more results of Lambert multipliers, we mention the upcoming
paper ~\cite{Cheshmavar}, in which the present authors showed that
the set of all Lambert multipliers acting between $L^p$-spaces are
commutative Banach algebra.

 In the following sections, Lambert multipliers acting between Orlicz
spaces are characterized. Also, we give necessary and sufficient
condition on $*$-multiplication operators to be closed range.
Finally, Fredholmness of the corresponding $*$-multiplication
operators is investigated.

%===========================================================================

\section{characterization of lambert multipliers}

Let $\Phi$ and $\Psi$ be Young functions. Define
$K^{*}_{\Phi,\Psi}$, the set of all Lambert multipliers from
$L^{\Phi}(\Sigma)$ into $L^{\Psi}(\Sigma)$, as follows:
$$K^{*}_{\Phi,\Psi}:=\{u\in D(E):u*L^{\Phi}(\Sigma)\subseteq L^{\Psi}
(\Sigma)\}~.$$ $K^{*}_{\Phi,\Psi}$ is a vector subspace of $D(E)$.
Put $K^{*}_{\Phi,\Phi}=K^{*}_{\Phi}$. In the following theorem we
characterize the members of $K^{*}_{\Phi}$.

\begin{thm}
Let $\Phi$ be Young function and $u\in D(E)$. Then $u\in
K^{*}_{\Phi}$ if and
 only if $E(\Phi(|u|))\in L^{\infty}(\mathcal{A})$.
\end{thm}
\begin{proof}
Let $E(\Phi(|u|))\in L^{\infty}(\mathcal{A})$ and $f\in
L^{\Phi}(\Sigma)$. Since $\Phi(|E(u)|)\leq E(\Phi(|u|))\leq
\|E(\Phi(|u|))\|_{\infty}$ a.e., that is, $|E(u)|\leq
\Phi^{-1}(\|E(\Phi(|u|))\|_{\infty})$, it follows ( Prop. 3 page
60 in \cite{Rao}) that,
  $$\int_{X}\Phi(\frac{|E(u)f|}{\Phi^{-1}(\|E(\Phi(|u|))\|_{\infty})\|f\|_{\Phi}})
  d\mu  \leq \int_{X}\Phi(\frac{|f|}{\|f\|_{\Phi}})d\mu \leq 1~.$$
Hence $\|E(u)f\|_{\Phi}\leq
\Phi^{-1}(\|E(\Phi(|u|))\|_{\infty})\|f\|_{\Phi}$. A similar
argument, using the fact that $E(fE(g)) = E(f)E(g)$, we also have
$$\int_{X}\Phi(\frac{|uE(f)|}{\Phi^{-1}(\|E(\Phi(|u|))\|_{\infty})
\|f\|_{\Phi}}) d\mu  \leq
\int_{X}\Phi(\frac{|f|}{\|f\|_{\Phi}})d\mu \leq 1~.$$ Thus
$\|E(u)E(f)\|_{\Phi}\leq \|uE(f)\|_{\Phi}\leq
\Phi^{-1}(\|E(\Phi(|u|)) \|_{\infty})\|f\|_{\Phi}$. Now, we get
that $$\|u*f\|_{\Phi} \leq
\|E(u)f\|_{\Phi}+\|uE(f)\|_{\Phi}+\|E(u)E(f)\|_{\Phi}\leq
3\Phi^{-1} (\|E(\Phi(|u|))\|_{\infty})\|f\|_{\Phi}~.$$ It follows
that $u*f\in L^{\Phi}(\Sigma)$, hence $u\in K^{*}_{\Phi}$.

Now, let $u\in K^{*}_{\Phi}$. A straightforward application of the
closed graph theorem shows that the operator $T_u :
L^{\Phi}(\Sigma) \rightarrow L^{\Phi}(\Sigma)$ given by $T_uf =
u*f$ is bounded. Define a linear functional $\Lambda$ on
$L^1(\mathcal{A})$ by
$$\Lambda(f)=\int_X E(\Phi(|u|))f d\mu,\quad\quad f\in L^1(\mathcal{A})~.$$
We show that $\Lambda$ is bounded.  Since $\Phi$ satisfies
$\bigtriangleup_{2}$-condition, we have (by the Corollary 5 page
26 in \cite{Rao}):
\begin{align*}
  |\Lambda(f)| &\leq \int_{X}E \left( \Phi(|u|) \right)|f| d\mu
               =\int_{X}E \left( \Phi(|u|)|f| \right) d\mu\\
               &\leq\int_{X} E\left(C|u|^{\alpha}|f|\right)d\mu
               =\int_{X} E\left(C\left(|u||f|^{\frac{1}{\alpha}}\right)^{\alpha}\right)d\mu\\
               &= \int_{X}C\left(|u||f|^{\frac{1}{\alpha}}\right)^{\alpha}d\mu
               = C\|T_{u}|f|^{\frac{1}{\alpha}}\|^{\alpha}_{\alpha}\\
               &\leq C\|T_{u}\|^{\alpha}\||f|^{\frac{1}{\alpha}}\|^{\alpha}_{\alpha}
               = C\|T_{u}\|^{\alpha}\|f\|_{1},
\end{align*}
 for some $\alpha>1, C>0$. Consequently, $\Lambda$ is a bounded linear
functional on $L^1(\mathcal{A})$ and
 $\|\Lambda\|\leq C\|T_u\|^{\alpha}$. By the Riesz representation theorem, there
  exists a unique function $g \in L^{\infty}(\mathcal{A})$ such that
$$\Lambda(f)=\int_X gf d\mu,\quad\quad f\in L^1(\mathcal{A})~.$$
Therefore, $g = E(\Phi(|u|))$ a.e. on $X$ and hence
 $E(\Phi(|u|))\in L^{\infty}(\mathcal{A})$~.
\end{proof}
Let $m$ be the collection $\{T_u : u\in K^{*}_{\Phi}\}$. An easy
consequence of the closed graph theorem shows that $m$ consist of
continuous linear transformations. Since $T_u T_v=T_{u*v}$ (it is
obvious), $m$ is commutative algebra. Similar argument as in the
proof of Theorem 4.1 in \cite{Alan.4}, $m$ is maximal abelian and
hence it is norm closed.\\ For $u\in K^{*}_{\Phi}$, we define its
norm by $\|u\|_{K^{*}_{\Phi}}:=
\Phi^{-1}(\|E(\Phi(|u|)\|_{\infty})$  such that
$(K^{*}_{\Phi},\|u\|_{K^{*}_{\Phi}})$ is respected as a normed
space. The next result reads as follows:

\begin{thm}
Let $u\in K^{*}_{\Phi}$, then the following holds:
\begin{itemize}
\item [(i)] $\|u\|_{K^{*}_{\Phi}}\leq \|T_u\|\leq
3\|u\|_{K^{*}_{\Phi}}$,
 \item [(ii)] $(K^{*}_{\Phi},\|u\|_{K^{*}_{\Phi}})$ is a
Banach space.
\end{itemize}
\end{thm}
\begin{proof}
 In order to prove (i), assume that $u\in K^{*}_{\Phi}$ and $f\in L^1(\mathcal{A})$.
 Then, $E(\Phi(|u|))\in L^{\infty}(\mathcal{A})$, and
$$\|E(\Phi(|u|)\|_{\infty}=\sup_{\|f\|\leq1}\int_{X}E \left( \Phi(|u|) \right)|f| d\mu
\leq C\|T_u\|^{\alpha}~.$$ That is,
$\Phi^{-1}(\|E(\Phi(|u|)\|_{\infty})\leq \|T_u\|$. It follows
that $\|u\|_{K^{*}_{\Phi}}\leq \|T_u\|$. On the other hand, by the
properties of conditional expectation operators, it is easy to see
that for each $f \in L^{\Phi}(\Sigma)$
 with $\|f\|_{\Phi} \leq 1$,
$$\max \{\|E(u)f\|_{\Phi},\|uE(f)\|_{\Phi}, \|E(u)E(f)\|_{\Phi} \}\leq \Phi^{-1}
(\|E(\Phi(|u|)\|_{\infty})~,$$ and so $\|T_u\|\leq
3\|u\|_{K^{*}_{\Phi}}$.

For the proof of (ii), assume that $\{u_n\}_{n=1}^{\infty}$  be a
Cauchy sequence with respect to the norm $\|.\|_{K^{*}_{\Phi}}$.
Let $f \in L^{\Phi}(\Sigma)$ and $g \in L^{\Psi}(\Sigma)$ be
arbitrary elements, then
$$|\int_X T_{u_n-u_m}(f)\overline{g}d\mu|\leq 3\|u_n-u_m\|_{K^{*}_{\Phi}}
\|f\|_{\Phi}\|g\|_{\Psi}~,$$ that is, $\{T_{u_n}\}_{n=1}^{\infty}$
is a Cauchy sequence in the weak operator topology. The subalgebra
$m$ is maximal abelian and so it is weakly closed. Therefore,
$\{T_{u_n-u_0}\}_{n=1}^{\infty}$ is weakly convergent
 to zero, for some $u_0 \in K^{*}_{\Phi}$. By the dominated convergence theorem
 we have
$$\int_{X}\lim_{n\rightarrow\infty}(u_n-u_0)fgd\mu=\lim_{n\rightarrow\infty}\int_{X}
T_{u_n-u_0}(f)gd\mu=0~.$$ Thus,
$\lim_{n\rightarrow\infty}(u_n-u_0)=0$, a.e. on $X$ and since $E$
is a contraction map, then,
$\lim_{n\rightarrow\infty}E(\Phi(|u_n-u_0|))=0$, a.e. on $X$.
Finally, $\|u_n-u_0\|_{K^{*}_{\Phi}}\rightarrow\infty$
  as $n\rightarrow\infty$.
\end{proof}

%===========================================================================

\section{fredholm $*$-multiplication operators}

In the following theorem, we establish a condition for a
$*$-multiplication operator $T_u$ to have closed range. We use the
symbols $\mathcal{N}(T_u)$ and $\mathcal{R}(T_u)$ to denote the
kernel and the range of $T_u$, respectively.

\begin{thm}
Let $u\in K^{*}_{\Phi}$. Then $T_u$ is closed range if and only if
there exists $\delta > 0$ such that $E\left(\Phi(|u|)\right) >
\delta$, almost everywhere on the support of $E(u)$.
\end{thm}
\begin{proof}
Let $S:=\{x\in X: E(u)(x)\neq 0\}$ be the support of $E(u)$. If
$T_u$ has closed range, then it is bounded below on $L^{\Phi}(S)$,
i.e., there exists a constant $k
> 0$ such that
$$\|T_uf\|_{\Phi}\geq k\|f\|_{\Phi},\quad\quad f\in L^{\Phi}(S)~.$$

Let $\delta=k/2$, and put $U := \{x \in S: E(\Phi(|u|))(x) <
\delta\}$~. Suppose on contrary $\mu(U)> 0$. Since
$(X,\mathcal{A},\mu|_{\mathcal{A}})$ is a $\sigma$-finite measure
space, we can find a set $B \in \mathcal{A}$ such that $Q := B\cap
S\subseteq U$ with $0<\mu(Q)<\infty$~. Then the
$\mathcal{A}$-measurable characteristic function $\chi_{Q}$ lies
in $L^{\Phi}(S)$. It is known that,
 $\|\chi_{Q}\|_{\Phi}=\frac{1}{\Phi^{-1}(1/\mu(Q))}$ and
\begin{align*}
  \|T_u\chi_{Q}\| &=\inf \{\varepsilon : \int_{S}\Phi(|u\chi_{Q}|/\varepsilon)
  d\mu\leq1\} \\
                  &<\inf \{\varepsilon : \int_{S}\Phi(|\Phi^{-1}(\delta)
                  \chi_{Q}|/\varepsilon)d\mu\leq1\} \\
                  &=\|\Phi^{-1}(\delta)\chi_{Q}\|_{\Phi}=\Phi^{-1}(\delta)
                  \|\chi_{Q}\|_{\Phi}~,
\end{align*}
which is a contradiction. Therefore, $\mu(U) = 0$, i.e.,
 $E\left(\Phi(|u|)\right) > \delta$ a.e. on $S$.
Conversely, suppose $E\left(\Phi(|u|)\right) > \delta$ a.e. on $S$
and
 $\{T_uf_n\}_{n=0}^{\infty}$ be an arbitrary
sequence in $\mathcal{R}(T_u)$, such that $\|T_uf_n -
g\|_{\Phi}\rightarrow 0$\,\ as $n \rightarrow \infty$, for
 some $g \in L^{\Phi}(\Sigma)$~. Hence
$$E(T_uf_n) = E(u)E(f_n)\xrightarrow{L^\Phi(\Sigma)} E(g),\,\
 \mbox{as}\,\ n \rightarrow \infty~.$$
Since, $E(1/\Phi(|u|))\chi_{S}=(1/E(\Phi(|u|)))\chi_{S}$~, then we
have
 $\Phi(E(1/|u|))\chi_{S}\leq 1/\delta$, and so we get that
 $E(1/|u|)\leq \Phi^{-1}(1/\delta)$ a.e. on $S$. Therefore, we have
\begin{align*}
  \|\frac{E(g)}{E(u)}\chi_{S}\|_{\Phi} &=\inf\{\varepsilon>0 :
   \int_{S}\Phi(|\frac{E(g)}{E(u)}\chi_{S}|/\varepsilon)d\mu\leq 1\}\\
   &=\inf\{\varepsilon>0 : \int_{S}\Phi(|E(g)E(\frac{1}{u})\chi_{S}|/\varepsilon)
   d\mu\leq 1\}\\
   &\leq\inf\{\varepsilon>0 : \int_{S}\Phi(|E(g)E(\frac{1}{|u|})\chi_{S}|/\varepsilon)
   d\mu\leq 1\}\\
   &\leq\inf\{\varepsilon>0 : \int_{S}\Phi(|\frac{|E(g)|}{\Phi^{-1}(\delta)
   }\chi_{S}|/\varepsilon)d\mu\leq 1\}\\
   &\leq \frac{1}{\Phi^{-1}(\delta)}\|g\|_{\Phi}~.
\end{align*}
This follows that $\frac{E(g)}{E(u)}\chi_{S}\in L^{\Phi}(S)$. Consequently,
$$E(f_n)\xrightarrow{L^\Phi(\Sigma)}\frac{E(g)}{E(u)}\chi_{S},\,\
 \mbox{as}\,\ n \rightarrow \infty ~,$$
and so there exist $f \in L^\Phi(\Sigma)$ such that,
$$f_n\xrightarrow{L^\Phi(\Sigma)}f,
%\left\{g+E(g)-\frac{uE(g)}{E(u)} \right\}
%\frac{\chi_{S}}{E(u)}:=f,
\,\ \mbox{as}\,\ n \rightarrow \infty~.$$ Thus $T_uf_n
\xrightarrow{L^\Phi(\Sigma)} T_uf, \,\
 \mbox{as}\,\ n \rightarrow \infty$, and hence $g =
T_uf$, which implies that $T_u$ is closed range.
\end{proof}
Recall (See \cite{Zaanen}) that $T_u$ is said to be a Fredholm
operator if $\mathcal{R}(T_u)$ is closed, $dim\mathcal{N}(T_u)
<\infty$, and $codim\mathcal{R}(T_u) < \infty$. Also, recall that
an $\mathcal{A}$-atom of the measure $\mu$ is an element $A\in
\mathcal{A}$ with $\mu(A)>0$ such that for each $F\in \Sigma$, if
$F\subseteq A$ then either $\mu(F)=0$ or $\mu(F)=\mu(A)$. A
measure with no atoms is called non-atomic. It is a well-known
fact that every $\sigma$-finite measure space
$(X,\mathcal{A},\mu|_\mathcal{A})$ can be partitioned uniquely as
$$X=(\bigcup_{n\in \mathbb{N}}A_n )\cup B~,$$
where $\{A_n\}_{n\in \mathbb{N}}$ is a countable collection of
pairwise disjoint $\mathcal{A}$-atoms and $B$, being disjoint from
each $A_n$, is non-atomic .\\

The following Theorem is a generalization of Prop. 3 due to Takagi
in \cite{Takagi}:

\begin{thm}\label{t2}
Let $u\in K^{*}_{\Phi}$ and $\mathcal{A}$ is a non-atomic measure
space. If the operator $T_u$ is Fredholm on $L^{\Phi}(\Sigma)$,
then we have $|E(u)|\geq \delta$, almost everywhere on $X$, for
some $\delta > 0$.
\end{thm}
\begin{proof}

Suppose that $T_u$ is a Fredholm operator. We first claim that
$T_u$ is onto. Suppose the contrary, pick $g \in
L^{\Phi}(\Sigma)\setminus\mathcal{R}(T_u)$. Since
$\mathcal{R}(T_u)$ is closed, $\Phi$ and $\Psi$ are complementary
Young functions and $\Phi$ satisfies $\bigtriangleup_2$-condition,
it follows from \cite{Rao} ( Corollary 5, page 77 and Theorem 7
page 110 ) that we can find a function $g^*\in
L^{\Psi}=(L^{\Phi})^*$, such that
$$\text{(I)}\quad\quad \int_{X} {g}g^*d\mu=1~,$$
and
$$\text{(II)}\quad\quad \int_{X} g^*{T_uf}d\mu=0,\ f\in L^{\Phi}
(\Sigma)~.$$ Now $\text{(I)}$ yields that the set $B_r = \{x\in
X:|E({g}g^*)(x)| \geq r\}$ has positive measure for some $r>0$. As
$\mathcal{A}$ is non-atomic, we can choose a sequence $\{A_n\}$ of
subsets of $B_r$ with $0<\mu(A_n)<\infty$ and $A_m \bigcap A_n
=\emptyset$ for $m\neq n$. Put $g^*_n=\chi_{A_n}g^*$. Clearly,
$g^*_n\in L^{\Psi}(\Sigma)$ and is nonzero, because
$$\int_{X}|{g}g^*_n|d\mu \geq \int_{A_n}|{g}g^*_n|d\mu=\int_{A_n}
E(|{g}g^*|)d\mu\geq \int_{A_n}|E({g}g^*)|d\mu\geq r\mu(A_n)>0~,$$
for each $n$. Also, for each $f \in L^{\Phi}(\Sigma)$,
$\chi_{A_n}f\in L^{\Phi}(\Sigma)$ and so $\text{(II)}$ implies
that
$$\int_{X} T^*_ug^*_n{f}d\mu=\int_{X} g^*_n{T_uf}d\mu=\int_{A_n} g^*{T_uf}d\mu
 = \int_{X}g^*{T_u(\chi_{A_n} f)}d\mu=0~,$$ which implies that $T^*_ug^*_n=0$
and so $g^*_n \in \mathcal{N}(T^*_u)$~. Since all the sets in
$\{A_n\}_n$ are disjoint, the sequence $\{g_n\}_n$ forms a
linearly independent subset of $\mathcal{N}(T^*_u)$. This
contradicts the fact that $\dim N(T^*_u)=\text{codim}
\mathcal{R}(T_u)<\infty$. Hence $T_u$ is onto. Let $Z(E(u)) :=
\{x\in X: E(u)(x) = 0\}$. Then $\mu(Z(E(u)))=0$. Since, if
$\mu(Z(E(u)))>0$, then there is an $F\subseteq Z(E(u))$ with
$0<\mu(F)< \infty$. If $\chi_F\in R(T_u)$, then there exists $f\in
L^{\Phi}(\Sigma)$ such that $T_uf = \chi_F$. Then
$$\mu(F)=\int_X\chi_Fd\mu=\int_F T_ufd\mu=\int_F E(T_uf)d\mu=\int_F E(u)E(f)d\mu=0~,$$
and this is a contradiction. So $\chi_F\in L^{\Phi}(\Sigma)
\setminus \mathcal{R} (T_u)$, which contradicts the fact that
$T_u$ is onto. For each $n=1,2,...,$ let
$$H_n=\{x\in X: \frac{\|E(\Phi(|u|))\|_{\infty}}{(n + 1)^2} < \Phi(|E(u)|)(x)
\leq \frac{\|E(\Phi(|u|))\|_{\infty}}{n^2} \}~,$$ and $H=\{n \in
\mathbb{N}:\mu(H_n)>0\}$. Then the $H_n$'s are pairwise disjoint,
$X=\bigcup_{n=1}^{\infty}H_n$ and $\mu(H_n)<\infty$ for each
$n\geq 1$. Take
\begin{equation*}
f(x) = \left\{
\begin{array}{rl}
     |E(u)|\Phi^{-1}(1/\mu(H_n))&x\in H_n, n\in H\\
        &  \\
     0 \quad\quad\quad\quad &\text{otherwise.}
\end{array} \right.
\end{equation*}
Then
\begin{align*}
  \int_X \Phi\left(\frac{|f(x)|}{\|E(\Phi(|u|))\|_{\infty}}\right)d\mu
   & =\sum_{n=1}^{\infty}\left(\int_{H_n}\Phi\left(\frac{\Phi(|E(u)|)(x)}
   {\|E(\Phi(|u|))\|_{\infty}}
   \Phi^{-1}\left(\frac{1}{\mu(H_n)}\right)\right)d\mu\right) \\
   &\leq \sum_{n=1}^{\infty}\left(\int_{H_n}\Phi\left(\frac{1}{n^2}
   \Phi^{-1}\left(\frac{1}{\mu(H_n)}\right)\right)d\mu\right) \\
   & \leq \sum_{n=1}^{\infty}\frac{1}{n^2}\frac{1}{\mu(H_n)}\int_{H_n}d\mu
   =\sum_{n=1}^{\infty}\frac{1}{n^2}<\infty~.
\end{align*}
Therefore, $f \in L^\Phi(\mathcal{A})$ and so there exist $g \in
L^\Phi(\Sigma)$
 such that $T_ug=f$. Hence
$E(u)E(g)=E(T_ug)=f$. Since $E(g) = f/E(u)$ except of $Z(E(u))$
and $\mu(Z(E(u)))=0$,
 it follows that
\begin{align*}
  \int_X\Phi(|g|)d\mu= \int_XE(\Phi(|g|))d\mu &\geq \int_X\Phi(|E(g)|)d\mu \\
                 &= \int_X\Phi(\frac{f}{|E(u)|})d\mu=\sum_{n\in H}\frac{1}{\mu(H_n)}
                  =\sum_{n\in H}1~.
\end{align*}
This implies that $H$ must be a finite set. So there is an $n_0$ such that
 $n \geq n_0$ implies $\mu(H_n)=0$. Together with $\mu(Z(E(u)))=0$, we obtain
$$\mu\left(\left\{x\in X:\Phi(|E(u)|)(x)\leq \frac{\|E(\Phi(|u|))\|_{\infty}}
{n_0^2}\right\}\right)= \mu\left(\bigcup_{n=n_0}^{\infty}H_n\cup
Z(E(u))\right)=0~,$$ that is $|E(u)|\geq
\Phi^{-1}\left(\|E(\Phi(|u|))\|_{\infty}/n_0^2\right)
 :=\delta$ almost everywhere on $X$.

\end{proof}

\begin{exam}
Let $X=[-3,3]$, $d\mu=dx$, let $\Sigma$ be the Lebesgue sets, and
$\mathcal{A}$ the $\sigma$-subalgebra generated by the sets
symmetric about the origin. Put $0<a\leq 3$. Let
$\Phi(x)=(1+|x|)\log(1+|x|)-|x|$. It is verified that $\Phi\in
\bigtriangleup_{2}$. For each $f\in L^{\Phi}(\Sigma)$ we have
\begin{align*}
  \int^{a}_{-a}E(f)(x)dx &= \int^{a}_{-a}f(x)dx\\
                         &= \int^{a}_{-a}\left\{\frac{f(x)+f(-x)}{2}+\frac{f(x)-f(-x)}{2}\right\}dx\\
                         &= \int^{a}_{-a}\frac{f(x)+f(-x)}{2}dx.
\end{align*}

Consequently, $(Ef)(x)=\frac{f(x)+f(-x)}{2}$. Now, if we take $u(x)=x^{4} +\sin x +3$,
then the $*$-multiplication operator $T_u$  has the form
$$(T_u)(x)=\left(x^{4}+\frac{1}{2}\sin x+3\right)f(x)+\left(\frac{1}{2}\sin x\right) f(-x).$$
Direct computation shows that $|E(u)| \geq 3$. Therefore, $T_u$ is a Fredholm  operator.
\end{exam}

%========================================================================

\end{document}